\documentclass[12pt]{amsart}
\usepackage{amscd,amsmath,amsthm,amssymb}
\usepackage{pstcol,pst-plot,pst-3d}
\usepackage{color}
\usepackage{pstricks}
\usepackage{stmaryrd}
\newpsstyle{fatline}{linewidth=1.5pt}
\newpsstyle{fyp}{fillstyle=solid,fillcolor=verylight}
\definecolor{verylight}{gray}{0.97}
\definecolor{light}{gray}{0.9}
\definecolor{medium}{gray}{0.85}
\definecolor{dark}{gray}{0.6}

 \def\NZQ{\mathbb}               
 \def\NN{{\NZQ N}}
 
 \def\ZZ{{\NZQ Z}}

 \def\frk{\mathfrak}               

 \def\mm{{\frk m}}

 \def\G{{\mathcal G}}

 \def\opn#1#2{\def#1{\operatorname{#2}}} 
 \opn\chara{char} \opn\length{\ell} \opn\pd{pd} \opn\rk{rk}
 \opn\projdim{proj\,dim} \opn\injdim{inj\,dim} \opn\rank{rank}
 \opn\depth{depth} \opn\grade{grade} \opn\height{height}
 \opn\embdim{emb\,dim} \opn\codim{codim}
 
 \opn\Tr{Tr} \opn\bigrank{big\,rank}
 \opn\superheight{superheight}\opn\lcm{lcm}
 \opn\trdeg{tr\,deg}
 \opn\reg{reg} \opn\lreg{lreg} \opn\ini{in} \opn\lpd{lpd}
 \opn\size{size} \opn\sdepth{sdepth}
 \opn\link{link}\opn\fdepth{fdepth}\opn\lex{lex}
 \opn\tr{tr}
 \opn\type{type}
 \opn\div{div} \opn\Div{Div} \opn\cl{cl} \opn\Cl{Cl}
 \opn\Spec{Spec} \opn\Supp{Supp} \opn\supp{supp} \opn\Sing{Sing}
 \opn\Ass{Ass} \opn\Min{Min}\opn\Mon{Mon}
 \opn\Ann{Ann} \opn\Rad{Rad} \opn\Soc{Soc}
 \opn\Im{Im} \opn\Ker{Ker} \opn\Coker{Coker} \opn\Am{Am}
 \opn\Hom{Hom} \opn\Tor{Tor} \opn\Ext{Ext} \opn\End{End}
 \opn\Aut{Aut} \opn\id{id}
 
 \opn\nat{nat}
 \opn\pff{pf}
 \opn\Pf{Pf} \opn\GL{GL} \opn\SL{SL} \opn\mod{mod} \opn\ord{ord}
 \opn\Gin{Gin} \opn\Hilb{Hilb}\opn\sort{sort}
 \opn\PF{PF}\opn\Ap{Ap}
 \opn\aff{aff} \opn
 \con{conv} \opn\relint{relint} \opn\st{st}
 \opn\lk{lk} \opn\cn{cn} \opn\core{core} \opn\vol{vol}  \opn\inp{inp} \opn\nilpot{nilpot}
 \opn\link{link} \opn\star{star}\opn\lex{lex}\opn\set{set}
 \opn\width{wd}
 \opn\Fr{F}
 \opn\QF{QF}
 \opn\G{G}
 \opn\type{type}\opn\res{res}
 \opn\gr{gr}
 
 \def\pot#1#2{#1[\kern-0.28ex[#2]\kern-0.28ex]}

 \opn\dirlim{\underrightarrow{\lim}}
 \opn\inivlim{\underleftarrow{\lim}}
 \let\to=\rightarrow
 
 \def\Implies{\ifmmode\Longrightarrow \else
         \unskip${}\Longrightarrow{}$\ignorespaces\fi}
 \def\implies{\ifmmode\Rightarrow \else
         \unskip${}\Rightarrow{}$\ignorespaces\fi}
 \def\iff{\ifmmode\Longleftrightarrow \else
         \unskip${}\Longleftrightarrow{}$\ignorespaces\fi}

 \let\:=\colon
 \newtheorem{Theorem}{Theorem}[section]
 \newtheorem{Lemma}[Theorem]{Lemma}
 \newtheorem{Corollary}[Theorem]{Corollary}
 \newtheorem{Proposition}[Theorem]{Proposition}
 \newtheorem{Remark}[Theorem]{Remark}
 
 \newtheorem{Example}[Theorem]{Example}

 \let\epsilon\varepsilon
 \let\kappa=\varkappa
 \def\qed{\ifhmode\textqed\fi
       \ifmmode\ifinner\quad\qedsymbol\else\dispqed\fi\fi}
 \def\textqed{\unskip\nobreak\penalty50
        \hskip2em\hbox{}\nobreak\hfil\qedsymbol
        \parfillskip=0pt \finalhyphendemerits=0}
 \def\dispqed{\rlap{\qquad\qedsymbol}}
 
 \opn\dis{dis}
 \def\pnt{{\raise0.5mm\hbox{\large\bf.}}}
 
 \opn\Lex{Lex}

\usepackage{tikz}
\newcommand{\divs}[1]{\textrm{div}(#1)}

\begin{document}

\title{Monomial ideals with tiny squares}
\author{Shalom Eliahou, J\"urgen Herzog and Maryam Mohammadi Saem}

\address{Shalom Eliahou, Univ. Littoral C\^ote d'Opale, EA 2597 - LMPA - Laboratoire de Math\'ematiques Pures et Appliqu\'ees Joseph Liouville, F-62228
 Calais, France and CNRS, FR 2956, France} \email{eliahou@univ-littoral.fr}

\address{J\"urgen Herzog, Fachbereich Mathematik, Universit\"at Duisburg-Essen, Campus Essen, 45117
Essen, Germany} \email{juergen.herzog@uni-essen.de}

\address{Maryam Mohammadi Saem, Faculty of Science, University of Mohaghegh Ardabili, P.O. Box 179, Ardabil, Iran}
\email{m.mohammadisaem@yahoo.com}

\begin{abstract} Let $I \subset K[x,y]$ be a monomial ideal. How small can $\mu(I^2)$ be in terms of $\mu(I)$? It has been expected that the inequality
$\mu(I^2) > \mu(I)$ should hold whenever $\mu(I) \ge 2$. Here we disprove this expectation and provide a somewhat surprising answer to the above question.
\end{abstract}

\maketitle

\section{Introduction}
For an ideal $I$ in a Noetherian ring $R$, let $\mu(I)$ denote as usual the least number of generators of $I$. If $\mu(I)=m$, how small can
$\mu(I^2)$ be in terms of $m$? Obviously, in suitable rings with zero-divisors, we may have $\mu(I^2)=0$. There even exist one-dimensional local domains $(R,\mm)$ with the property that the square of their maximal ideal $\mm$ requires less generators than $\mm$ itself, see \cite{HW,O}. However, if $R$ is a regular local ring, or if $R$ is a polynomial ring over a field $K$ and $I$ is a homogeneous ideal of $R$, it has been expected in \cite{HMZ} that the inequality $\mu(I^2) > \mu(I)$ should hold whenever $\mu(I) \ge 2$. This is indeed the case for any integrally closed
ideal $I$ in a $2$-dimensional regular local ring. On the other hand, it is not  too difficult to construct examples of monomial ideals $I$ in a polynomial ring $S$ with at least 4 variables such that $\mu(I^2)<\mu(I)$. However, these examples satisfy $\height I<\dim S$. So far no ideals $I$ with  $\mu(I^2)<\mu(I)$ were known for $2$-dimensional regular  rings. In this paper, we shall prove the following statements.

 \begin{Theorem}\label{nine}
 For every integer $m \geq 5$, there exists a monomial ideal $I\subset K[x,y]$ such that $\mu(I)=m$ and $\mu(I^2) = 9$.
 \end{Theorem}

Moreover, this result is best possible for $m \ge 6$.
 \begin{Theorem}\label{ge nine} Let $I\subset K[x,y]$ be a monomial ideal. If $\mu(I) \ge 6$ then $\mu(I^2) \ge 9$.
 \end{Theorem}

Here are some notation to be used throughout. We denote by $\mathcal{M}$ the set of monomials in $K[x,y]$, i.e.
$$
\mathcal{M} = \{x^iy^j \mid i,j \in \NN\}.
$$
As usual, we view $\mathcal{M}$ as partially ordered by divisibility.

For a monomial ideal $J \subset K[x,y]$, we denote by $\mathcal{G}(J)$ its unique minimal system of monomial generators. It is well known that
$\mathcal{G}(J)$ is of cardinality $\mu(J)$ and consists of all monomials in $J$ which are minimal under divisibility, i.e.
$$
\mathcal{G}(J) = \big(\mathcal{M} \cap J\big) \setminus \big(\mathcal{M} \cap J\big)\mathcal{M}^*
$$
where $\mathcal{M}^*=\mathcal{M} \setminus \{1\}$.

\smallskip
Finally, given integers $a \le b$, we denote by $[a,b]$ the \emph{integer interval} they span, i.e.
$$
[a,b] = \{c \in \ZZ \mid a \le c \le b\}.
$$

\section{Preliminaries}

Let $m \ge 2$ be an integer and let $I \subset K[x,y]$ be a  monomial ideal such that $\mu(I)=m$. Then $\mathcal{G}(I)=\{u_1,\dots,u_m\}$, where
$u_i=x^{a_i}y^{b_i}$ for all $i$ and where the exponents $a_i,b_i \in \NN$ satisfy
\begin{equation}\label{monomials ui}
\begin{array}{lcr} \vspace{0.15cm}
& a_1 > a_2 > \dots > a_m, & \\
& b_1 < b_2 < \dots < b_m. &
\end{array}
\end{equation}
Removing any common factor among the $u_i$, we may assume $a_m=b_1=0$ if desired.

\smallskip
Let $V = \{(i,j) \in \NN^2 \mid 1 \le i \le j \le m\}$, and consider the map
$$
\begin{array}{rccc}
f \colon & V & \to & I^2 \\
& (i,j) & \mapsto & u_iu_j.
\end{array}
$$
Then $f(V)$ generates $I^2$, but not minimally so in general. Thus $\mathcal{G}(I^2) \subseteq f(V)$, with equality only occurring in special
circumstances. How small  can $\mathcal{G}(I^2)$ be? So far this was not well understood. We provide a complete answer in this note.

\smallskip
We partially order $\NN^2$ as follows:
$$
(i,j) \le (k,l) \iff i \le k \mbox{ and } j \le l.
$$
Since $V \subset \NN^2$, we also view $V$ as partially ordered by $\le$. Recall that $\mathcal{M}$ is partially ordered by divisibility. These orderings
interact in a simple yet useful way, as shown by the lemma below. Recall that an \emph{antichain} in a poset is a subset whose elements are pairwise
noncomparable.

\begin{Lemma}\label{antichain} Let $v, v' \in V$. If $f(v)$ divides $f(v')$, then either $v = v'$ or else $\{v,v'\}$ is an antichain in $V$.
\end{Lemma}

Equivalently, \emph{if $v < v'$, then $\{f(v),f(v')\}$ is an antichain in $\mathcal{M}$.}

\begin{proof} Set $v=(i,j), \, v'=(k,l)$, and assume $v < v'$. Then $i \le k$ and $j \le l$, and at least one of these inequalities is strict. We have
$f(v)=u_iu_j=x^{a_i+a_j}y^{b_i+b_j}$. Similarly, $f(v')= x^{a_k+a_l}y^{b_k+b_l}$. Since the $a_t$ are decreasing and the $b_t$ are increasing, and since
either $i<k$ or $j<l$, we have
\begin{eqnarray*}
a_i+a_j & > & a_k+b_l, \\
b_i+b_j & < & b_k+b_l.
\end{eqnarray*}
Therefore, neither $u_iu_j$ divides $u_ku_l$, nor conversely.
\end{proof}

\begin{Remark}\label{noncomparable}
If $v,v' \in V$, then $v,v'$ are noncomparable if and only if
$$
\min v < \min v' \le \max v' < \max v \,\, \textrm{ or }\,\,  \min v' < \min v \le \max v < \max v'.
$$
\end{Remark}
Picturing any pair $w=(i,j) \in V$ as an edge joining the vertices $i, j \in \NN$, it is useful to keep in mind that graphically, noncomparable pairs $v,v'
\in V$ look like this:

\medskip
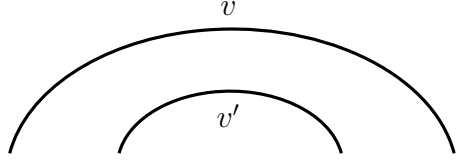
\begin{center}
\begin{tikzpicture}[xscale=1.5]
\draw [very thick] (0,0) arc [radius=2, start angle=10, end angle=170];
\node at (-2,1.7) [above] {$v$};
\draw [very thick] (-1,0) arc [radius=1, start angle=10, end angle=170];
\node at (-2,0.8) [below] {$v'$};
\end{tikzpicture}
\end{center}

\bigskip

For instance, if $v=(1,3)$, the only element of $V$ which is noncomparable to $v$ is $v'=(2,2)$, as is graphically obvious. Therefore by
Lemma~\ref{antichain}, if $u_1u_3 \notin \mathcal{G}(I^2)$, then the only element of $\mathcal{G}(I^2)$ dividing $u_1u_3$ must be $u_2^2$.

Of course, noncomparability is not a transitive relation. For instance, the only elements of $V$ which are noncomparable to $v'=(2,2)$ are all $v'' \in V$
of the form $v''=(1,j)$ with $j \in [3,m]$. Therefore, if $u_2^2 \notin \mathcal{G}(I^2)$, any minimal generator of $I^2$ dividing it must be of the form
$u_1u_j$ for some \textit{a priori} unspecified $j \in [3,m]$.

\medskip
We shall also need the next lemma.
\begin{Lemma}\label{interval}
Let $v,v_1,v_2 \in V$. Assume that $v_1 \le v_2$ and that $f(v)$ divides both $f(v_1), f(v_2)$. Then $f(v)$ divides $f(v')$ for all $v' \in V$ such that
$v_ 1\le v' \le v_2$.
\end{Lemma}
\begin{proof} Set $v=(i,j), v_1=(i_1,j_1), v_2=(i_2,j_2), v'=(r,s)$. The hypotheses, together with the respective monotonicity of the $a_t$ and $b_t$,
imply
\begin{eqnarray*}
a_i+a_j & \le & a_{i_2}+a_{j_2}, \\
b_i+b_j & \le & b_{i_1}+b_{j_1}.
\end{eqnarray*}
Since $i_1 \le r \le i_2$ and $j_1 \le s \le j_2$, it follows that
\begin{eqnarray*}
a_{i_2}+a_{j_2} & \le & a_{r}+a_{s}, \\
b_{i_1}+b_{j_1} & \le & b_{r}+b_{s}.
\end{eqnarray*}
Thus $a_i+a_j \le a_{r}+a_{s}$ and $b_i+b_j \le b_{r}+b_{s}$, whence $f(v)$ divides $f(v')$.
\end{proof}

\section{Conditions for tiny squares}
We now give conditions on a monomial ideal $I=(u_1,\dots,u_m)$ in $K[x,y]$ which will force $\mu(I^2)$ to be a small constant.

\begin{Theorem}\label{conditions} Let $m \ge 5$. Let $I =(u_1,\dots,u_m) \subset K[x,y]$ be a monomial ideal with $u_i =x^{a_i}y^{b_i}$ for all $i$, where
$a_1> \dots > a_m$ and $b_1 < \dots < b_m$. Assume that the following divisibility conditions hold:
\begin{eqnarray}
u_1 u_m &|& u_2 u_{m-1} \label{1,m} \\
u_1 u_{m-1} & | & u_2 u_3, \, u_{m-2}^2 \label{1,m-1} \\
u_2^2 & | & u_1 u_3, \, u_1 u_{m-2} \label{2,2} \\
u_2 u_m & | & u_3 u_{m-1}, \, u_{m-2} u_{m-1} \label{2,m} \\
u_{m-1}^2 & | & u_3u_m, \, u_{m-2}u_m \label{m-1,m-1}.
\end{eqnarray}
Then $\mu(I^2)=9$. More precisely, $I^2$ is minimally generated by the set
$$
G=\{u_1^2, u_1u_2, u_2^2\} \cup \{u_1u_{m-1}, u_1u_{m}, u_2u_{m}\} \cup \{u_{m-1}^2, u_{m-1}u_m, u_m^2\}.
$$
\end{Theorem}
\begin{proof} To see that $G$ generates $I^2$, we must show that, for all $i,j \in [1,m]$ such that $i \le j$, the monomial $u_iu_j$ is a multiple of some
element in $G$ . We distinguish several cases.

\smallskip
\textbf{Case $i=1.$} If $j \in \{1,2,m-1,m\}$, we are done since the corresponding monomial $u_1u_j$ belongs to $G$. Assume now $j \in [3,m-2]$. Then
$(1,3) \le (1,j) \le (1,m-2)$, and since $u_2^2$ divides both $u_1 u_3$ and $u_1 u_{m-2}$ by \eqref{2,2}, it also divides $u_1u_j$ by Lemma~\ref{interval}.

\smallskip
\textbf{Case $i,j \in [2,m-2].$} If $(i,j) = (2,2)$, we are done since $u_2^2 \in G$. If $(i,j) \not= (2,2)$, then $(2,3) \le (i,j) \le (m-2,m-2)$, and
since $u_1u_{m-1}$ divides both $u_2 u_3$ and $u_{m-2}^2$ by \eqref{1,m-1}, it follows from Lemma~\ref{interval} that $u_1u_{m-1}$ also divides $u_iu_j$.

\smallskip
\textbf{Case $j=m-1$.} If $i=2$, then  by \eqref{1,m}, $u_2u_{m-1}$ is divisible by $u_1u_m$ and $u_1u_m \in G$. If $i \in [3,m-2]$, then $(3,m-1) \le
(i,m-1) \le (m-2,m-1)$, and since $u_2u_m$ divides both $u_3 u_{m-1}$ and $u_{m-2} u_{m-1}$ by \eqref{2,m}, it also divides $u_iu_{m-1}$ by
Lemma~\ref{interval}. Finally, if $i=m-1$, we are done since $u_{m-1}^2 \in G$.

\smallskip
\textbf{Case $j=m.$} If $i=2$, we are done since $u_2u_{m} \in G$. If $i \in [3,m-2]$, then $(3,m) \le (i,m) \le (m-2,m)$, and since $u_{m-1}^2$ divides
both $u_3 u_{m}$ and $u_{m-2} u_{m}$ by \eqref{m-1,m-1}, it also divides $u_iu_{m}$ by Lemma~\ref{interval}. Finally, if $i \in \{m-1,m\}$, we are done
since $u_{m-1}u_m, u_m^2 \in G$.

\smallskip
We conclude that $G$ generates $I^2$. To establish the equality $G=\mathcal{G}(I^2)$, it remains to see that $G$ is an antichain. Let $A \subset V$ consist
of nine pairs $(i,j)$ such that $G=\{u_iu_j \mid (i,j) \in A\}$, for instance
$$
A = \{(1,1),(1,2),(2,2),(1,m-1),(1,m),(2,m),(m-1,m-1),(m-1,m),(m,m)\}.
$$
Replacing $(2,2)$ by $(1,3)$ and $(m-1,m-1)$ by $(m-2,m)$ yields a chain in $V$, namely
\begin{eqnarray*}
& & (1,1)<(1,2)<(1,3)<  \\
& & (1,m-1)<(1,m)<(2,m)< \\
& & (m-2,m)<(m-1,m)<(m,m).
\end{eqnarray*}
It follows from Lemma~\ref{antichain} that the set
$$G'=G\cup\{u_1u_3,u_{m-2}u_{m}\}\setminus \{u_2^2,u_{m-1}^2\}$$
is an antichain. In particular, its subset $G\setminus \{u_2^2,u_{m-1}^2\}$ also is. Consider now $u_1u_3$. Since $G$ generates $I^2$, then $u_1u_3$ is a
multiple of some element of $G$. Moreover, since $G'$ is an antichain, the only possible factors of $u_1u_3$ in $G$ are $u_2^2$ and $u_{m-1}^2$. But since
$(1,3) < (m-1,m-1)$, Lemma~\ref{antichain} implies that $\{u_1u_3,u_{m-1}^2\}$ is an antichain. Therefore, the only factor of $u_1u_3$ in $G$ is $u_2^2$.
Hence, since any monomial in $I^2$ has at least one factor in $G \cap \mathcal{G}(I^2)$, it follows that $u_2^2 \in \mathcal{G}(I^2)$.

Let us now see that $u_2^2$ cannot divide another member of $G$. By Lemma~\ref{antichain}, the only possible multiples of $u_2^2$ in $G$ would be
$u_1u_{m-1}$ and $u_1u_m$. But neither possibility arises, since $u_1u_{m-1}$ divides $u_2u_3$ by \eqref{1,m-1}, and $u_1u_m$ divides $u_2u_{m-1}$ by
\eqref{1,m}.

Entirely symmetric arguments apply to $u_{m-1}^2$. We conclude that $G = \mathcal{G}(I^2)$, as desired.
\end{proof}

\section{An explicit construction}
We now show that the conditions of Theorem~\ref{conditions} ensuring $\mu(I^2)=9$ are realizable.
\begin{Proposition}\label{crazy ideal} Let $m\ge 5$. Let $I=(u_1,\dots,u_m)$, where $u_i=x^{a_i}y^{b_i}$ with exponents $a_i,b_i$ defined as follows:
\begin{eqnarray*}
(a_1,\dots,a_m) & = & (5m, 4m, 4m-1, \dots,3m+4,m,0), \\
(b_1,\dots,b_m) & = & (a_m,\dots,a_1).
\end{eqnarray*}
Then $\mu(I)=m$ and $\mu(I^2)=9$.
\end{Proposition}
\begin{proof} Since the $a_i$ are decreasing and the $b_i$ are increasing, the $u_i$ constitute an antichain and hence a minimal system of generators of
$I$. In order to show $\mu(I^2)=9$, it suffices to prove that the divisibility conditions of Theorem~\ref{conditions} are met. This is straightforward. For
convenience, here are the monomials involved:
$$
\begin{array}{lll}
u_1 u_m  =  x^{5m}y^{5m} &  u_2 u_{m-1}  =  x^{5m}y^{5m} & \\
u_1 u_{m-1}  =  x^{6m}y^{4m} &  u_2 u_3  =  x^{8m-1}y^{4m+4}  & u_{m-2}^2  =  x^{6m+8}y^{8m-2}  \\
u_2^2  =  x^{8m}y^{2m} & u_1 u_3  =  x^{9m-1}y^{3m+4} & u_1 u_{m-2}  =  x^{8m+4}y^{4m-1} \\
u_2 u_m  =  x^{4m}y^{6m} &  u_3 u_{m-1}  =  x^{5m-1}y^{7m+4} & u_{m-2} u_{m-1}  =  x^{4m+4}y^{8m-1} \\
u_{m-1}^2  =  x^{2m}y^{8m} & u_3u_m  =  x^{4m-1}y^{8m+4} & u_{m-2}u_m  =  x^{3m+4}y^{9m-1}.
\end{array}
$$
\end{proof}

\medskip
Let us now look at the degree distribution of $\mathcal{G}(I)$ and of $\mathcal{G}(I^2)$. The generators $u_i$ of $I$ are of degree $a_i+b_i=a_i+a_{m+1-i}$
for all $i \in [1,m]$. For $m \ge 5$, we find
$$
\begin{array}{rcl}
\deg(u_1)=\deg(u_2)=\deg(u_{m-1})=\deg(u_m) & = & 5m, \\
\deg(u_i)=\deg(u_{m+1-i}) & = & 7m+3
\end{array}
$$
for all $3 \le i \le m-2$. That is, the ideal $I$ is generated in two degrees only. For $I^2$ the situation is even simpler. It is generated in the single
degree $10m$, the common degree of its nine minimal generators.

\begin{Example}
Here is the case $m=10$ of Proposition~\ref{crazy ideal}. We have
$$
I = (x^{50}, \, x^{40}y^{10}, \, x^{39}y^{34}, \, x^{38}y^{35}, \, x^{37}y^{36}, \, x^{36}y^{37}, \, x^{35}y^{38}, \, x^{34}y^{39}, \, x^{10}y^{40}, \,
y^{50})
$$
and
$$
I^2 = (x^{100}, \, x^{90}y^{10}, \, x^{80}y^{20}, \, x^{60}y^{40}, \, x^{50}y^{50}, \, x^{40}y^{60}, \, x^{20}y^{80}, \, x^{10}y^{90}, \, y^{100}).
$$
\end{Example}

 \medskip
 To conclude this section, let us show that in contrast, if $I$ is generated in a single degree, then $\mu(I^2)$ grows to infinity with $\mu(I)$.

\begin{Proposition}\label{single degree} Let $I \subset K[x,y]$ be a monomial ideal generated in a single degree. If $\mu(I)=m$ then $\mu(I^2) \ge 2m-1$.
\end{Proposition}
\begin{proof} Set $\mathcal{G}(I)=\{u_1,\dots,u_m\}$, and assume $\deg(u_i)=d$ for all $i$. Since $\deg(u_iu_j)=2d$ for all $i,j$, the distinct $u_iu_j$
form an antichain for divisibility. Hence $\mu(I^2)$ equals the number of pairwise distinct $u_iu_j$. Now that number is at least $2m-1$, as witnessed by
the subset
$$
u_1u_1, \, u_1u_2 , \, \dots , \, u_1u_m , \, u_2u_m , \, \dots , \, u_mu_m.
$$
\end{proof}

\section{Computing $\mu(I^k)$}

For the ideal $I$ satisfying $\mu(I)=m, \mu(I^2)=9$ of the preceding section, we now determine $\mu(I^k)$ for all $k$.

\begin{Proposition}\label{I^k}
Given $m \ge 5$, let $I=(u_1,\dots,u_m)$ be the ideal defined in Proposition \ref{crazy ideal}. For all $k \ge 3$, we have
$\mu(I^k)=5k+1.$
\end{Proposition}

\begin{proof} It follows from Theorem~\ref{conditions} that $I^2$ is minimally generated by
$$
G=\{u_1^2, u_1u_2, u_2^2\} \cup \{u_1u_{m-1}, u_1u_{m}, u_2u_{m}\} \cup \{u_{m-1}^2, u_{m-1}u_m, u_m^2\}.
$$
Let $J=(u_1,u_2,u_{m-1},u_m)$ and $H=(u_3,\dots,u_{m-2})$, so that $I=J+H$. Then $I^2=J^2$, since $I^2 = (G) \subseteq J^2 \subseteq I^2$. Hence $IJ, IH \subseteq J^2$, from which it follows that $I^k=J^k$ for all $k \ge 3$. 

Now, by construction in Proposition~\ref{crazy ideal}, we have $J=(x^{5m},x^{4m}y^m,x^my^{4m},y^{5m})$. Let $J_0=(x^{5},x^{4}y,xy^{4},y^{5})$. Then $\mu(J^k)=\mu(J_0^k)$ for all $k \ge 1$. Moreover, we have 
$$J_0^3=(x^{15}, x^{14}y,\dots,xy^{14},y^{15})$$
as easily seen. It follows that $J_0^k=(x,y)^{5k}$ for all $k\ge 3$. Therefore
$$
\mu(I^k)=\mu(J^k)=\mu(J_0^k)=5k+1
$$
for all $k \ge 3$, as stated.
\end{proof}

\begin{Corollary}
For any integer $k_0\geq2$, there exists a monomial ideal $I$ in $K[x,y]$ such that $\mu(I^k)<\mu(I)$ for all $k$ such that $2 \le k \le k_0$.
\end{Corollary}
\begin{proof} Let $m > 5k_0+1$. By Propositions~\ref{crazy ideal} and \ref{I^k}, there is an ideal $I$ with $\mu(I)=m > 5k_0+1$ and satisfying $\mu(I^k)\le 5k+1$ for all $k \ge 2$. Therefore $\mu(I^k)<\mu(I)$ for all $2 \le k \le k_0$.
\end{proof}

\section{Optimality}

 We now prove that Theorem~\ref{conditions} is best possible, in the sense that $\mu(I^2)$ is bounded below by $9$ if $\mu(I) \ge 6$. For that we need some
 more notation. If $I \subset K[x,y]$ is a monomial ideal with $\mathcal{G}(I)=\{u_1,\dots,u_m\}$, let us denote
 $$
 \mathcal{G}^2(I)=\{u_iu_j \mid 1 \le i \le j \le m\},
 $$
 i.e. $\mathcal{G}^2(I)=f(V)$ in the notation of the Introduction. Let
$$\gamma \colon \mathcal{G}^2(I) \to \mathcal{G}(I^2)$$ be the map defined, for all $u_iu_j \in \mathcal{G}^2(I)$, by
\begin{equation}\label{gamma}
\gamma(u_iu_j)=u_ku_l,
\end{equation}
where $u_ku_l \in \mathcal{G}(I^2)$ is the \emph{lexicographically first} minimal monomial generator of $I^2$ dividing $u_iu_j$. Note that $\gamma \circ
\gamma = \gamma$. Finally, given $u_iu_j \in \mathcal{G}^2(I)$, we denote
 $$
 \div(u_iu_j) = \{u_ru_s \in \mathcal{G}^2(I) \mid u_ru_s \textrm{ divides } u_iu_j\}.
 $$
Of course $\gamma(u_iu_j) \in \div(u_iu_j)$. We will repeatedly use the following reformulation of Lemma~\ref{antichain}.

 \begin{Lemma}\label{div} For all $(i,j) \in V$, we have
$$
 \div(u_iu_j) \subseteq \{u_iu_j\} \cup \{u_ru_s \in \mathcal{G}^2(I) \mid (r,s) \textrm{ is noncomparable to } (i,j) \textrm{ in } V\}.
$$
 \end{Lemma}
 \begin{proof} As stated in Lemma~\ref{antichain}, if $u_ru_s$ divides $u_iu_j$ and $(r,s) \not= (i,j)$, then $\{(r,s),(i,j)\}$ is an antichain in $V$.
 \end{proof}

Thus for instance, $\div(u_1u_3) \subseteq \{u_1u_3,u_2^2\}$ as already observed after Remark~\ref{noncomparable}. Again, it is useful to keep in mind the
picture representing noncomparable pairs following that remark.

\smallskip
Here is the main result of this section.

 \begin{Theorem}\label{ge nine} Let $I\subset K[x,y]$ be a monomial ideal. If $\mu(I) \ge 6$ then $\mu(I^2) \ge 9$.
 \end{Theorem}

  \begin{proof}{} Let $m=\mu(I)$. As before, we denote $\mathcal{G}(I) = \{u_1,\dots, u_m\}$, where $u_i=x^{a_i}y^{b_i}$ for all $i$ and where the
  $a_i,b_i$ are decreasing and increasing, respectively. We start by considering the following six monomials in $\mathcal{G}^2(I)$:
  $$
  u_1^2,u_1u_2,u_1u_3,u_{m-2}u_{m},u_{m-1}u_{m},u_m^2.
  $$

 By Lemma~\ref{div}, we have $\gamma(u_1^2)=u_1^2$ since $(1,1)$ is comparable to every $(i,j) \in V$, namely $(1,1) \le (i,j)$. Similarly, each one of the
 pairs $(1,2), (m-1,m), (m,m)$ is comparable to all elements of $V$, whence
 $$\gamma(u_1u_2)=u_1u_2, \,\, \gamma(u_{m-1}u_m)=u_{m-1}u_m, \,\, \gamma(u_m^2)=u_m^2.
 $$
Consider $u_1u_3$. Then
 $$\gamma(u_1u_3) \in \divs{u_1u_3} \subseteq \{u_1u_3,u_2^2\}$$ by Lemma~\ref{div}, since $(1,3)$ is comparable to all elements of $V$ except $(2,2)$.
 Symmetrically, we have
$$\gamma(u_{m-2}u_m) \in \{u_{m-2}u_m, u_{m-1}^2\}.$$
Thus so far, we have six pairwise distinct minimal monomial generators of $I^2$, namely
\begin{equation}\label{so far 6}
A=\{u_1^2,u_1u_2,\gamma(u_1u_3),\gamma(u_{m-2}u_m),u_{m-1}u_m,u_m^2\} \subseteq \mathcal{G}(I^2).
\end{equation}
Considering $u_1u_m$, we distinguish two cases.

\medskip
\noindent
\textbf{Case 1: $\gamma(u_1u_m)\not=u_1u_m$.} Therefore $u_1u_m \notin \mathcal{G}(I^2)$, and $\gamma(u_1u_m) = u_ru_s$ for some $r,s \in [2,m-1]$ such
that $r \le s$. Since $\div(u_2u_3) \cap \div(u_{m-2}u_{m-1}) \subseteq \{u_1u_m\}$ by Lemma~\ref{div}, it follows that $\gamma(u_2u_3) \not=
\gamma(u_{m-2}u_{m-1})$. Moreover, it also follows from Lemma~\ref{div} that $\gamma(u_2u_3), \gamma(u_{m-2}u_{m-1}) \notin A$. So let
$$
A'=A \sqcup\{\gamma(u_2u_3), \gamma(u_{m-2}u_{m-1})\}.
$$
Then $\textrm{card}(A')=8$. There are two subcases.

\textbf{Case (1.1)} If $\gamma(u_1u_m) \notin A'$, then we are done, since joining $\gamma(u_1u_m)$ to $A'$ yields nine minimal generators of $I^2$.

\textbf{Case (1.2)} If $\gamma(u_1u_m) \in A'$, then by Lemma~\ref{div}, the only possibilities are as follows:
$$
\begin{array}{lllll}
\gamma(u_1u_m) & = & u_2^2 & = & \gamma(u_1u_3), \\
\gamma(u_1u_m) & = & u_2u_3 & = & \gamma(u_2u_3), \\
\gamma(u_1u_m) & = & u_{m-2}u_{m-1} & = & \gamma(u_{m-2}u_{m-1}), \\
\gamma(u_1u_m) & = & u_{m-1}^2 & = & \gamma(u_{m-1}^2).
\end{array}
$$
But then in each case, Lemma~\ref{div} implies $\gamma(u_2u_{m-2}) \notin A'$, and we are done again.

\medskip
\noindent
\textbf{Case 2: $\gamma(u_1u_m)=u_1u_m$.} Then $u_1u_m \in \mathcal{G}(I^2)$. Let
$$
A'=A \sqcup\{u_1u_m\}=\{u_1^2,u_1u_2,\gamma(u_1u_3),\gamma(u_{m-2}u_m),u_{m-1}u_m,u_m^2,u_1u_m\} .
$$
Then $A' \subseteq \mathcal{G}(I^2)$ and $\textrm{card}(A')=7$. Comparing $\gamma(u_{1}u_{m-1})$ to $\gamma(u_{1}u_{3})$, and $\gamma(u_{2}u_{m})$ to
$\gamma(u_{m-2}u_{m})$, gives rise to four subcases. Let us consider them in turn.

\smallskip
\textbf{Case (2.1): $\gamma(u_{1}u_{m-1})=\gamma(u_{1}u_{3})$, $\gamma(u_{2}u_{m})=\gamma(u_{m-2}u_{m})$.}

Since $\div(u_{1}u_{m-1}) \cap \div(u_{1}u_{3}) \subseteq \{u_2^2\}$, we have $\gamma(u_{1}u_{m-1})=\gamma(u_{1}u_{3}) = u_2^2$. Similarly, we have
$\gamma(u_{2}u_{m})=\gamma(u_{m-2}u_{m})=u_{m-1}^2$. Therefore
\begin{eqnarray*}
u_2^2 & | & u_{1}u_{m-1}, \\
u_{m-1}^2 & | & u_2u_m.
\end{eqnarray*}
Hence $u_2^2 u_{m-1}^2$ $|$ $u_{1}u_{m-1}u_2u_m$, implying $u_2 u_{m-1}$ $|$ $u_{1}u_m$. Whence
\begin{equation}\label{u_1u_m}
u_2 u_{m-1}=u_{1}u_m
\end{equation}
since $u_1u_m \in \mathcal{G}(I^2)$. Let
$$
A''=A' \cup\{\gamma(u_2u_3), \gamma(u_{m-2}u_{m-1})\}.
$$
Then $A'' \subseteq \mathcal{G}(I^2)$ by construction. We claim that $\textrm{card}(A'') = 9$. Indeed, Lemma~\ref{div} implies
$$
\{\gamma(u_2u_3), \gamma(u_{m-2}u_{m-1})\} \cap A' \subseteq \{u_1u_m\}.
$$
But since $u_1u_m=u_2u_{m-1}$ here, as observed in \eqref{u_1u_m}, it follows that
\begin{equation}\label{u2u3}
u_1u_m \notin \div(u_2u_3) \cup \div(u_{m-2}u_{m-1}),
\end{equation}
whence the above intersection is empty. Moreover, since $\div(u_2u_3) \cap \div(u_{m-2}u_{m-1}) \subseteq \{u_1u_m\}$, it follows from \eqref{u2u3} that
$\gamma(u_2u_3) \not= \gamma(u_{m-2}u_{m-1})$. Therefore $\textrm{card}(A'') = 9$ as claimed, and we are done in this case.

\smallskip
\textbf{Case (2.2): $\gamma(u_{1}u_{m-1})=\gamma(u_{1}u_{3})$, $\gamma(u_{2}u_{m})\not=\gamma(u_{m-2}u_{m})$.}

This is by far the most delicate case. As observed in the preceding case, the equality $\gamma(u_{1}u_{m-1})=\gamma(u_{1}u_{3})$ implies
$\gamma(u_{1}u_{m-1})=\gamma(u_{1}u_{3})=u_2^2$. Now $\gamma(u_{2}u_{m}) \notin A'$, as follows from Lemma~\ref{div} and the inequality
$\gamma(u_{2}u_{m})\not=\gamma(u_{m-2}u_{m})$. Let
$$
A''=A' \cup\{\gamma(u_{2}u_{m})\}.
$$
Then $A'' \subseteq \mathcal{G}(I^2)$ and $\textrm{card}(A'') = 8$. We distinguish four subcases, according to the values of $\gamma(u_{2}u_{m})$ and
$\gamma(u_{m-2}u_{m})$. Recall that $\gamma(u_{m-2}u_{m}) \in \{u_{m-2}u_{m}, u_{m-1}^2\}$.

\smallskip
\textbf{Case (2.2.1): $\gamma(u_{2}u_{m})=u_{2}u_{m}$, $\gamma(u_{m-2}u_{m})=u_{m-2}u_{m}$.}

Then
$$
A''=\{u_1^2,u_1u_2,u_2^2,u_{m-2}u_m,u_{m-1}u_m,u_m^2,u_1u_m, u_2u_m\} .
$$

Interestingly here, for $m=5$, there are explicit cases where $A'' = \mathcal{G}(I^2)$. Now under our hypothesis $m \ge 6$, a ninth element in
$\mathcal{G}(I^2) \setminus A''$ is given by $\gamma(u_3u_m)$, as easily follows from the present shape of $A''$ and Lemma~\ref{div}.

\smallskip
\textbf{Case (2.2.2): $\gamma(u_{2}u_{m})=u_{2}u_{m}$, $\gamma(u_{m-2}u_{m})=u_{m-1}^2$.}

Then
$$
A''=\{u_1^2,u_1u_2,u_2^2,u_{m-1}^2,u_{m-1}u_m,u_m^2,u_1u_m, u_2u_m\} .
$$

Consider $u_2u_{m-1}$. If $\gamma(u_2u_{m-1}) \notin A''$, we have our ninth minimal generator and we are done. Assume now for a contradiction that
$\gamma(u_2u_{m-1}) \in A''$. Then under the present shape of $A''$ and Lemma~\ref{div}, the only possibility is $\gamma(u_2u_{m-1})=u_1u_m$. Therefore
\begin{eqnarray*}
u_2^2 & | & u_{1}u_{m-1}, \\
u_1u_{m} & | & u_2u_{m-1}.
\end{eqnarray*}
Hence $u_1u_2^2u_{m} \, | \, u_{1}u_2u_{m-1}^2$, implying $u_2u_m \, | \, u_{m-1}^2$. This is a contradiction, since $u_{m-1}^2 = \gamma(u_{m-2}u_{m})$,
whence $u_{m-1}^2$ divides $u_{m-2}u_{m}$ which is not divisible by $u_2u_m$.

\smallskip
\textbf{Case (2.2.3): $\gamma(u_{2}u_{m})\not=u_{2}u_{m}$, $\gamma(u_{m-2}u_{m})=u_{m-2}u_{m}$.}

By Lemma~\ref{div}, we have $\gamma(u_{2}u_{m})=u_ru_s$ for some $3 \le r \le s \le m-1$. Then
$$
A''=\{u_1^2,u_1u_2,u_2^2,u_{m-2}u_m,u_{m-1}u_m,u_m^2,u_1u_m, u_ru_s\} .
$$

Consider $u_{m-3}u_m$. If $\gamma(u_{m-3}u_m) \notin A''$, we have our ninth minimal generator and we are done. If $\gamma(u_{m-3}u_m) \in A''$, then the
only possibility allowed by Lemma~\ref{div} is $\gamma(u_{m-3}u_m) =u_ru_s$. The same lemma then implies $m-2 \le r \le s \le m-1$.

Consider now $u_2u_3$.  If $\gamma(u_2u_3) \notin A''$, we are done. Assume now for a contradiction that $\gamma(u_2u_3) \in A''$. Then the only
possibility is $\gamma(u_2u_3)=u_1u_m$. Therefore, since $\gamma(u_1u_3)=u_2^2$ in the present case, we have
\begin{eqnarray*}
u_2^2 & | & u_{1}u_{3}, \\
u_1u_{m} & | & u_2u_{3}.
\end{eqnarray*}
Hence $u_1u_2^2u_{m} \, | \, u_{1}u_2u_{3}^2$, implying $u_2u_m \, | \, u_{3}^2$. This is a contradiction, since $\gamma(u_{2}u_{m}) = u_{r}u_s$, whence
$u_ru_s$ divides $u_{2}u_{m}$ but cannot divide $u_3^2$ since $r,s \in [m-2,m-1]$.

\smallskip
\textbf{Case (2.2.4): $\gamma(u_{2}u_{m})\not=u_{2}u_{m}$, $\gamma(u_{m-2}u_{m})=u_{m-1}^2$.}

As above, we have $\gamma(u_{2}u_{m})=u_ru_s$ for some $3 \le r \le s \le m-1$, and here with $(r,s) \not= (m-1,m-1)$ since
$\gamma(u_{2}u_{m})\not=\gamma(u_{m-2}u_{m})$ in the present Case (2.2). Then
$$
A''=\{u_1^2,u_1u_2,u_2^2,u_{m-1}^2,u_{m-1}u_m,u_m^2,u_1u_m, u_ru_s\} .
$$
To settle this case, it suffices to show that either $\gamma(u_2u_{m-1})$ or $\gamma(u_{m-2}u_{m-1})$ lies outside $A''$. Assume for a contradiction the
contrary, i.e.
\begin{equation}\label{contra}
\{\gamma(u_2u_{m-1}), \gamma(u_{m-2}u_{m-1})\} \subset A''.
\end{equation}
Lemma~\ref{div} then implies
\begin{equation}\label{cond1}
\{\gamma(u_2u_{m-1}), \gamma(u_{m-2}u_{m-1})\} \subseteq \{u_1u_m, u_ru_s\}.
\end{equation}
On the other hand, that same lemma implies
\begin{equation}\label{cond2}
\div(u_2u_{m-1}) \cap \div(u_{m-2}u_{m-1}) \subseteq \{u_1u_m\}.
\end{equation}

We cannot have $\gamma(u_{m-2}u_{m-1})=u_1u_m$. For otherwise, in the present Case (2.2.4), that equality would imply
\begin{eqnarray*}
u_1u_{m} & | & u_{m-2}u_{m-1}, \\
u_{m-1}^2 & | & u_{m-2}u_{m}.
\end{eqnarray*}
Therefore $u_1u_{m-1}^2u_{m} \, | \, u_{m-2}^2u_{m-1}u_{m}$, whence $u_1u_{m-1} \, | \, u_{m-2}^2$. This is impossible, since $u_2^2$ divides $u_1u_{m-1}$
in the present Case (2.2), whereas $u_2^2$ cannot divide $u_{m-2}^2$ by Lemma~\ref{div}.

Therefore $\gamma(u_{m-2}u_{m-1}) \not=u_1u_m$, whence $\gamma(u_{m-2}u_{m-1})=u_ru_s$ by \eqref{cond1}. Since $r,s \in [3,m-1]$, this and Lemma~\ref{div}
imply  $(r,s)=(m-2,m-1)$. That is, we have
\begin{equation}\label{u_ru_s}
\gamma(u_{m-2}u_{m-1}) = u_ru_s=u_{m-2}u_{m-1}.
\end{equation}

Consider now $u_2u_{m-1}$. It follows from \eqref{cond1}, \eqref{cond2} and \eqref{u_ru_s}, that $\gamma(u_2u_{m-1})=u_1u_m$. In the present Case (2.2),
this yields
\begin{eqnarray*}
u_1u_{m} & | & u_2u_{m-1}, \\
u_2^2 & | & u_{1}u_{m-1}.
\end{eqnarray*}
Therefore $u_2u_m \, | \, u_{m-1}^2$. But $u_{m-1}^2$ divides $u_{m-2}u_m$ in the present Case (2.2.4), yet $u_2u_m$ cannot divide $u_{m-2}u_m$. This
contradiction shows that \eqref{contra} is absurd. Therefore
$$
\textrm{card}\big(A'' \cup \{\gamma(u_2u_{m-1}), \gamma(u_{m-2}u_{m-1})\}\big) \ge 9.
$$
This settles Case (2.2.4) and concludes Case (2.2).

\smallskip
\textbf{Case (2.3): $\gamma(u_{1}u_{m-1})\not=\gamma(u_{1}u_{3})$, $\gamma(u_{2}u_{m})=\gamma(u_{m-2}u_{m})$.}

This case is symmetrical to Case (2.2). It follows from it by interchanging the variables $x,y$ and reversing the sequence $u_1,\dots,u_m$.

\smallskip
\textbf{Case (2.4): $\gamma(u_{1}u_{m-1})\not=\gamma(u_{1}u_{3})$, $\gamma(u_{2}u_{m})\not=\gamma(u_{m-2}u_{m})$.}

Being in Case (2), recall that
$$
A'=A \sqcup\{u_1u_m\}=\{u_1^2,u_1u_2,\gamma(u_1u_3),\gamma(u_{m-2}u_m),u_{m-1}u_m,u_m^2,u_1u_m\},
$$
a subset of cardinality $7$ in $\mathcal{G}(I^2)$. It follows from the present hypotheses that
\begin{equation}\label{empty}
A' \cap \{\gamma(u_{1}u_{m-1}), \gamma(u_{2}u_{m})\} = \emptyset.
\end{equation}
Thus, it suffices to show $\gamma(u_{1}u_{m-1}) \not= \gamma(u_{2}u_{m})$, for then this would imply
\begin{equation}\label{9}
\textrm{card}\big(A' \cup \{\gamma(u_{1}u_{m-1}), \gamma(u_{2}u_{m})\}\big) = 9
\end{equation}
and we would be done. Assume for a contradiction the contrary, i.e. that $\gamma(u_{1}u_{m-1}) = \gamma(u_{2}u_{m})$. Lemma~\ref{div} then implies
\begin{equation}\label{(1,m)}
\gamma(u_{1}u_{m-1}) = \gamma(u_{2}u_{m}) = u_ru_s
\end{equation}
for some $r,s \in [3,m-2]$ with $r \le s$. Thus $u_ru_s$ divides both $u_{1}u_{m-1}$ and $u_{2}u_{m}$. Since $(1,m-1) < (1,m) < (2,m)$ in $V$,
Lemma~\ref{interval} implies that $u_ru_s$ also divides $u_1u_m$. But since $u_1u_m \in A' \subseteq \mathcal{G}(I^2)$, it follows that $u_ru_s=u_1u_m$.
Thus $u_ru_s \in A'$, contrary to the conjunction of \eqref{(1,m)} and \eqref{empty}. This contradiction settles \eqref{9} and concludes the proof of the
present last case, and hence of the theorem.
\end{proof}

\medskip
 For completeness, let us give optimal lower bounds on $\mu(I^2)$ for monomial ideals $I \subset K[x,y]$ such that $\mu(I)=m \le 5$. The result is as
 follows.
$$
\mu(I^2) \ge \left\{
\begin{array}{cc}
1 & \textrm{ if } m=1, \\
3 & \textrm{ if } m=2, \\
5 & \textrm{ if } m=3, \\
7 & \textrm{ if } m=4, \\
8 & \textrm{ if } m=5.
\end{array}
\right.
$$
As stated, these five lower bounds are sharp. The verification is left to the reader.

\medskip
Let us conclude this paper with two questions. Let $m \ge 6$ be an integer, and let $I \subset K[x,y]$ be a monomial ideal such that $\mu(I)=m$. We have
seen in Proposition~\ref{single degree} that if $I$ is generated in a single degree, then $\mu(I^2) \ge 2m-1$. Our ideals in Proposition~\ref{crazy ideal}
reaching the absolute minimum $\mu(I^2) =9$ are generated in two degrees depending on $m$, namely $5m$ and $7m+3$.

Here is our first question. If $I$ is generated in two degrees $d_1 < d_2$ \emph{such that the difference $d=d_2-d_1$ is fixed and independent of $m$},
does it follow that $\mu(I^2)$ must grow to infinity with $m$? This seems to be true for $d=1$, but it would be nice to have a proof.

Our second question is, what would be the proper generalization of the present results in $n$ variables?

\end{document}